\documentclass[a4paper]{amsart}
\usepackage{amssymb,mathtools}
\usepackage{enumerate}
\usepackage[T1]{fontenc}
\usepackage{lmodern}
\usepackage{textcomp}
\usepackage{microtype}
\usepackage[english]{babel}
\usepackage{hyperref}
\newtheorem{theorem}{Theorem}[section]
\newtheorem{lemma}[theorem]{Lemma}
\newtheorem{corollary}[theorem]{Corollary}
\theoremstyle{definition}

\theoremstyle{remark}
\newtheorem{remark}[theorem]{Remark}
\numberwithin{equation}{section}
\newcommand{\RR}{\mathbb{R}}
\newcommand{\CC}{\mathbb{C}}
\newcommand{\NN}{\mathbb{N}}

\newcommand{\cT}{\mathcal{T}}
\newcommand{\cK}{\mathcal{K}}
\newcommand{\eps}{\varepsilon}
\newcommand{\euler}{\mathrm{e}}
\renewcommand{\restriction}{|}
\newcommand*\Diff[1]{\mathop{}\!\mathrm{d}#1}
\DeclarePairedDelimiter{\abs}{|}{|}
\DeclarePairedDelimiter{\norm}{\lVert}{\rVert}
\newcommand{\bes}{\begin{equation*}}
\newcommand{\ees}{\end{equation*}}
\newcommand{\be}{\begin{equation}}
\newcommand{\ee}{\end{equation}}
\newcommand{\eqs}[1]{\begin{align*}#1\end{align*}}

\title[Uncertainty principles in Gelfand-Shilov spaces]
{Uncertainty principles with error term in Gelfand-Shilov spaces}
\hypersetup{
	pdftitle={Note on uncertainty principles in Gelfand-Shilov spaces},
	pdfauthor={Alexander Dicke, Albrecht Seelmann},
	hidelinks,
}
\author[A.~Dicke]{Alexander Dicke}
\author[A.~Seelmann]{Albrecht Seelmann}
\address[A.D., A.S.]{
	Technische Univer\-si\-t\"at Dortmund,
	Germany
}
\email{\{alexander.dicke,albrecht.seelmann\}@mathematik.tu-dortmund.de}
\begin{document}
%
%
\begin{abstract}
	In this note, an alternative approach to establish observability for semigroups based on their smoothing properties is presented. 
	The results discussed here reproduce some of those recently obtained in [arXiv:2112.01788], but the current proof allows to
	get rid of several technical assumptions by following the standard complex analytic approach established by Kovrijkine combined
	with an idea from [arXiv:2201.02370].
\end{abstract}
\maketitle
%
%

\section{Introduction and results}

Uncertainty principles are frequently used in control theory to prove observability for certain abstract Cauchy problems. 
Often this is done via the so-called Lebeau-Robbiano method, where an uncertainty principle for elements in the spectral subspace, 
a so-called \emph{spectral inequality}, is combined with a \emph{dissipation estimate}, see \cite{TenenbaumT-11, BeauchardPS-18,
NakicTTV-20, GallaunST-20}.
The aforementioned spectral inequalities were studied for several differential operators, see, e.g., 
\cite{EgidiV-20, BeauchardJPS-21, EgidiS-21, MartinPS-20-arxiv, 
DickeSV-21, DickeSV-22, MoyanoL-19, NakicTTVS-20, DickeRST-20}
and the references cited therein. Suitable dissipation estimates to treat also semigroups generated by some quadratic differential
operators were provided in \cite{BeauchardPS-18, Alphonse-20, MartinPS-20-arxiv, DickeSV-22}.

A different approach was introduced in \cite{Martin-21}, based on \cite{Miller-10}. It allows to derive observability estimates
from uncertainty principles with error term established for functions in the range of the semigroup associated to the abstract
Cauchy problem. 
In the situation of \cite{Martin-21}, these uncertainty principles are established using Gelfand-Shilov smoothing effects. By the
latter we mean that for the strongly continuous contraction semigroup $(\cT(t))_{t\geq 0}$ on $L^2(\RR^d)$ there exist constants
$C \geq 1$, $t_0 \in (0,1)$, $\nu>0$ and $0<\mu\leq 1$ with $\nu+\mu\geq 1$, and $r_1 \geq 0$, $r_2 > 0$, such that for all
$g\in L^2(\RR^d)$ we have
\be\label{eq:smoothing}
	\norm{(1+\abs{x}^2)^{n/2}\partial^\beta\cT(t)g}_{L^2(\RR^d)}
	\leq
	\frac{C^{1+n+\abs{\beta}}}{t^{r_1+r_2(n+\abs{\beta})}} 
		(n!)^\nu(\abs{\beta}!)^\mu \norm{g}_{L^2(\RR^d)}
\ee
for all $t\in (0,t_0)$ and all $n\in\NN$, $\beta\in\NN_0^d$.

In this context we prove the following variant of \cite[Theorem~2.3]{Martin-21}.

\begin{theorem}\label{thm:uncertainty} 
	Suppose that $f \in C^\infty(\RR^d)$ satisfies
	\be\label{eq:smoothing-f}
		\norm{(1+\abs{x}^2)^{n/2}\partial^\beta f}_{L^2(\RR^d)}
		\leq
		D_1D_2^{n+\abs{\beta}}(n!)^{\nu}(\abs{\beta}!)^\mu
		,
		\quad
		n\in\NN_0,\beta\in\NN_0^d
		,
	\ee
	with some $D_1 > 0$, $D_2 \geq 1$, $\nu \geq 0$, and $0 \leq \mu < 1$.
	Moreover, let $\delta \in [0,1]$ with $s = \delta\nu+\mu < 1$,
	and let $\rho\colon \RR^d \to (0,\infty)$ be a measurable function satisfying 
	\be
		\rho(x) \leq R(1+\abs{x}^2)^{\delta/2}
		\quad 
		\text{for all}\quad
		x\in\RR^d
	\ee
	with some $R \geq 1$ and
	\be\label{eq:smallness}
		\rho(x) 
		\leq 
		\eta\abs{x}
		\quad 
		\text{for all}\quad
		\abs{x} \geq r_0
	\ee
	with some $\eta \in (0,1)$ and some $r_0 \geq 1$.
	
	Then, for every measurable set $\omega \subset \RR^d$ satisfying 
	\be\label{eq:sensor-set}
		\frac{\abs{B(x,\rho(x))\cap \omega}}{\abs{B(x,\rho(x))}}
		\geq 
		\gamma
		\quad \text{for all} \quad 
		x\in\RR^d
	\ee
	with some $\gamma\in (0,1)$ and for every $\eps \in (0,1]$, we have
	\be\label{eq:uncertainty}
		\norm{f}_{L^2(\RR^d)}^2
		\leq 
		\euler^{K\cdot 
		\bigr( 1 + \log\frac{1}{\eps} + D_2^{2/(1-s)}\bigr)}
		 \norm{f}_{L^2(\omega)}^2 + \eps D_1^2
		,
	\ee
	where $K \geq 1$ is a constant depending on $\gamma, R,r_0,\eta,\nu,s$, and the dimension $d$.
\end{theorem}

Here, $B(x,\rho(x))$ denotes the open Euclidean ball of radius $\rho(x) > 0$ centered at $x$.
Note that \eqref{eq:smallness} is automatically satisfied if $\delta < 1$ with, say, $\eta = 1/2$ and
$r_0 \geq (4R)^{1/(1-\delta)}$.

In \cite{Martin-21}, the same result is proved but under more technical
assumptions, namely that $\rho$ is a Lipschitz contraction with a uniform positive lower bound. On the other hand, the case
$s = 1$, which also allows $\mu = 1$, is treated in \cite{Martin-21} but is not in the scope of the method we discuss here.
However, \cite{Martin-21} does not present any application in terms of observability for this case.

Our proof, as well as the one in \cite{Martin-21}, follows the approach from \cite{Kovrijkine-01,Kovrijkine-thesis}. The main
idea of the latter is to localize certain Bernstein-type inequalities on so-called \emph{good} elements of some covering of
$\RR^d$. Since in the setting of Theorem~\ref{thm:uncertainty} there is no Bernstein-type inequality at disposal, the definition of
good elements replaces, in some sense, the missing Bernstein-type inequalities needed. The proof then reduces to a local estimate
for (quasi-)analytic functions. For the latter, \cite{Martin-21} uses an estimate for quasianalytic functions proven in 
\cite{NazarovSV-04}, see also the $L^2$-version in \cite[Proposition~5.10]{Martin-21}, and a suitable estimate for the so-called
\emph{Bang degree}. By contrast, we rely on the more standard approach for (complex) analytic functions from
\cite{Kovrijkine-01,Kovrijkine-thesis} by estimating suitable Taylor expansions. This is combined with ideas introduced in
\cite{DickeSV-21,DickeSV-22} that incorporate the quadratic decay guaranteed by \eqref{eq:smoothing-f} to reduce the
considerations to a bounded subset of $\RR^d$. This allows us to obtain a more streamlined proof while getting rid of the
mentioned technical assumptions in \cite{Martin-21}.

If $(\cT(t))_{t\geq 0}$ is a strongly continuous contraction semigroup on $L^2(\RR^d)$ satisfying the Gelfand-Shilov smoothing
effects \eqref{eq:smoothing}, then $f = \cT(t)g$ with $g \in L^2(\RR^d)$ and $t \in (0,t_0)$ satisfies \eqref{eq:smoothing-f}
with 
\[
	D_1 = \frac{C}{t^{r_1}}\norm{g}_{L^2(\RR^d)}
	\quad \text{and} \quad 
	D_2 = \frac{C}{t^{r_2}}
	.
\]
Thus, choosing the constant $\eps$ in Theorem~\ref{thm:uncertainty} appropriately, we are able to apply \cite[Lemma~2.1]{Miller-10} in 
literally the same way as in the proof of \cite[Theorem~2.11]{Martin-21} and thereby obtain the following observability result,
which reproduces \cite[Theorem~2.11]{Martin-21}. We omit the proof for brevity.

\begin{corollary}\label{cor:observability}
	Suppose that $(\cT(t))_{t\geq 0}$ is a strongly continuous contraction semigroup satisfying \eqref{eq:smoothing} with some
	constants $C \geq 1$, $\nu \geq 0$, 
	$0\leq \mu < 1$, $r_1\geq 0$, $r_2 > 0$, and let $\delta \in [0,1]$ and
	$\rho \colon \RR^d \to (0,\infty)$ be as in Theorem~\ref{thm:uncertainty}. 
	Then 
	\bes
		\norm{\cT(T)g}_{L^2(\RR^d)}^2 
		\leq 
		N\exp\biggl(\frac{N}{T^{\frac{2r_2}{1-s}}} \biggr)\int_0^T \norm{\cT(t)g}_{L^2(\omega)}^2 \Diff{t}
		,
		\quad 
		g \in L^2(\RR^d)
		,\
		T > 0
		,
	\ees
	for every measurable set $\omega\subset\RR^d$ satisfying \eqref{eq:sensor-set} with some $\gamma \in (0,1)$.
	Here, $N \geq 1$ is a constant depending on $\gamma, R,r_0,\eta,\nu,s,C,r_2$, 
	and the dimension $d$.
\end{corollary} 

As shown in \cite[Corollary~2.2]{Alphonse-20b} and \cite[Lemma~5.2]{Martin-21}, the semigroup generated by the fractional
(an-)isotropic Shubin operator $((-\Delta)^m + \abs{x}^{2k})^\theta$ with $k,m\in\NN$ and $\theta > 1/(2m)$ satisfies
\eqref{eq:smoothing} with 
\[
	\nu = \max\Bigl\{ \frac{1}{2k\theta} , \frac{m}{k+m} \Bigr\}
	\quad \text{and} \quad 
	\mu = \max\Bigl\{ \frac{1}{2m\theta} , \frac{k}{k+m} \Bigr\}
	.
\]
Hence, Corollary~\ref{cor:observability} can be applied, which reproduces \cite[Corollary~2.12]{Martin-21}. 
It should however be mentioned that in the general case $\nu = \mu = 1/2$ of
Corollary~\ref{cor:observability} (a particular instance being a Shubin operator with $k=m$ and $\theta=1$) stronger results than
Corollary~\ref{cor:observability} in terms of the conditions on $\omega$ are available, see \cite[Theorem~3.5]{DickeSV-22}.
More precisely, in this case the density $\gamma$ is allowed to be variable and exhibit a certain subexponential decay, so that
$\omega$ may even have finite measure. 
Note that semigroups satisfying \eqref{eq:smoothing} can also be studied using forced symmetrization (of $\mu$ and $\nu$), see \cite{Alphonse-20b}, 
but by this approach one in any case loses special characteristics of the particular operator at hand. 
In the present setting, we are able to obtain a variant of Theorem~\ref{thm:uncertainty} where the density
$\gamma$ is allowed to exhibit a polynomial decay, but the result seems not to be sufficient to give observability as in
Corollary~\ref{cor:observability}, see Theorem~\ref{thm:uncertainty-decay} and Remark~\ref{rem:ucp} below.
If even $k=m=1$, the Shubin operator corresponds to the harmonic oscillator, for which also a sharper observability constant can
be obtained. Indeed, \cite[Theorem~6.1]{DickeSV-22} shows that the observability constant can then be chosen to vanish as
$T\to\infty$. 
We expect that such results hold also for the general Shubin operators.
However, this would require setting up a spectral inequality for these operators, which seems out of reach at the moment

\subsection*{Acknowledgements} The authors are grateful to Ivan Veseli\'c for introducing them to this field of research
and to Christopher Strothmann who referred to the reference \cite{Olver-97} for the asymptotics of the 
series used in the proof of Theorem~\ref{thm:uncertainty}. 
Both authors have been partially supported by the DFG grant VE~253/10-1 entitled \emph{Quantitative unique continuation properties of elliptic PDEs with
variable 2nd order coefficients and applications in control theory, Anderson localization, and photonics}.

%
%

\section{Proof of Theorem~\ref{thm:uncertainty}}

Let $\eps \in (0,1]$, and choose
\be\label{eq:choice-r}
	r
	:=
	\frac{D_2}{\sqrt{\eps/2}}
	\geq
	1
	,
\ee
so that 
\[
	\sup_{x\in \RR^d \setminus B(0,r)} \frac{1}{1+\abs{x}^2} 
	\leq 
	\frac{\eps}{2D_2^2}
	.
\]
Then, \eqref{eq:smoothing-f} with $n = 1$ and $\beta = 0$ implies that
\be\label{eq:compact}
		\norm{f}_{L^2(\RR^d \setminus B(0,r))}^2
		\leq
		\frac{\eps}{2} \cdot \frac{\norm{(1+\abs{x}^2)^{1/2} f}_{L^2(\RR^d)}^2}{D_2^2} 
		\leq
		\frac{\eps D_1^2}{2} 
		.
\ee

We abbreviate $w(x) = w_\delta(x) = (1+\abs{x}^2)^{\delta/2}$. 
If $\delta < 1$, we infer from \cite[Lemma~5.3]{Martin-21} that \eqref{eq:smoothing-f} implies
\be\label{eq:smoothing-f-delta}
	\norm{w^n\partial^\beta f}_{L^2(\RR^d)}
	\leq
	D_1\tilde{D}_2^{n+\abs{\beta}}(n!)^{\delta\nu}(\abs{\beta}!)^\mu
	,
	\quad 
	n\in\NN_0,\beta\in\NN_0^d,
\ee
with $\tilde{D}_2 = 8^\nu\euler^\nu D_2 \geq 1$. If $\delta = 1$, then \eqref{eq:smoothing-f-delta} agrees with
\eqref{eq:smoothing-f} with $\tilde{D}_2 = D_2 \geq 1$. We therefore just work with \eqref{eq:smoothing-f-delta} for the
remaining part.

The proof of the theorem now follows the following lines: 
Inspired by \cite{EgidiS-21}, the estimates \eqref{eq:smoothing-f-delta} imply, in particular, that $f$ is analytic, see
Lemma~\ref{lem:analyticity}.
Moreover, by \eqref{eq:compact} the contribution of $f$ outside of the ball $B(0,r)$ can be subsumed into the error term. 
The ball $B(0,r)$ is then covered with balls of the form $B(x,\rho(x))$ by Besicovitch covering theorem. Based on
\eqref{eq:smoothing-f-delta} and following \cite{Kovrijkine-01,Kovrijkine-thesis} and \cite{Martin-21}, these balls are
classified into good and bad ones, where on good balls local Bernstein-type estimates are available and the contribution of bad
balls can again be subsumed into the error term, see Lemma~\ref{lem:good-bad}.
Following again \cite{Kovrijkine-01,Kovrijkine-thesis}, on good balls the local Bernstein-type estimates allow to bound suitable
Taylor expansions of $f$, which by analyticity of $f$ lead to local estimates of the desired form, see
Lemmas~\ref{lem:pointwise}--\ref{lem:M_k}.
Summing over all good balls finally concludes the proof.

We consider balls $B(x,\rho(x))$ with $B(x,\rho(x))\cap B(0,r) \neq \emptyset$. The latter requires that
$\abs{x}-\rho(x) < r$ and, thus, $\abs{x} < r_0$ or $(1-\eta)\abs{x} \leq \abs{x} - \rho(x) < r$, that is,
$\abs{x} < r/(1-\eta)$. By Besicovitch covering theorem, see, e.g., \cite[Theorem~2.7]{Mattila-99}, there is $\cK_0 \subset \NN$
and a collection of points $(y_k)_{k \in \cK_0}$ with $\abs{ y_k } < \max\{ r_0 , r/(1-\eta) \}$ such that the family of balls
$Q_k = B(y_k,\rho(y_k))$, $k \in \cK_0$, gives an essential covering of $B(0,\max\{ r_0 , r/(1-\eta) \})$ with overlap at most
$\kappa \geq 1$. Here, the proof in \cite[Theorem~2.7]{Mattila-99} and a simple calculation shows that $\kappa$ can be chosen as
$\kappa = K_{\mathrm{Bes}}^d$ with a universal constant $K_{\mathrm{Bes}} \geq 1$. With
$Q_0 := \RR^d \setminus \bigcup_{k \in\cK_0} Q_k$ and $\cK := \cK_0 \cup \{0\}$, the family $(Q_k)_{k \in \cK}$ thus gives an
essential covering of $\RR^d$ with overlap at most $\kappa = K_{\mathrm{Bes}}^d \geq 1$.

\subsection{Good and bad balls}

Similarly as in \cite{Martin-21}, we now define the so-called good elements of the covering. We do this in such a way that we
have some localized Bernstein-type inequality on all good elements. More precisely, we say that $Q_k$, $k \in \cK_0$, is
\emph{good} with respect to $f$ if for all $m \in \NN_0$ we have 
\[
	\sum_{\abs{\beta}=m} \frac{1}{\beta!}\norm{w^m \partial^\beta f}_{L^2(Q_k)}^2
	\leq
	\frac{2\kappa}{\eps} \cdot \frac{2^{m+1} d^m q_m^2}{m!} \norm{f}_{L^2(Q_k)}^2
	,
\]
where
\[
	q_m 
	= 
	\tilde{D}_2^{2m}(m!)^{\delta\nu+\mu}
	=
	\tilde{D}_2^{2m}(m!)^s
	.
\]
We call $Q_k$, $k \in \cK_0$, \emph{bad} if it is not good.

Although we can not show that the mass of $f$ on the good balls covers some fixed fraction of the mass of $f$ on the whole of
$\RR^d$, inequality \eqref{eq:smoothing-f-delta} nevertheless implies that the mass of $f$ on the bad balls is bounded by
$\eps D_1^2/2$. Hence, the contribution of the bad elements can likewise be subsumed into the error term. This is summarized in
the following result.

\begin{lemma}\label{lem:good-bad}
	We have
	\[
		\norm{f}_{L^2(\RR^d)}^2
		\leq
		\norm{f}_{L^2(\bigcup_{k \colon Q_k\, \mathrm{good}} Q_k)}^2 + \eps D_1^2
		.
	\]
\end{lemma}

\begin{proof}
	Since 
	\[
		\norm{f}_{L^2(\RR^d)}^2 
		\leq
		\norm{f}_{L^2(\bigcup_{k \colon Q_k\, \mathrm{good}} Q_k)}^2 
		+ 
		\norm{f}_{L^2(\bigcup_{k \colon Q_k\, \mathrm{bad}} Q_k)}^2
		+
		\norm{f}_{L^2(Q_0)}^2
		,
	\]
	it suffices to show that
	\be\label{eq:bad-contribution}
		\sum_{k \colon Q_k\, \mathrm{bad}} \norm{f}_{L^2(Q_k)}^2 + \norm{f}_{L^2(Q_0)}^2
		\leq
		\eps D_1^2
		.
	\ee
	To this end, we first note that $Q_0 \subset \RR^d \setminus B(0,r)$ and, thus,
	$\norm{f}_{L^2(Q_0)}^2 \leq \eps D_1^2/2$ by estimate \eqref{eq:compact}.
	Let now $Q_k$, $k \in \cK_0$, be bad, that is, there is $m \in \NN_0$ such that
	\eqs{
		\norm{f}_{L^2(Q_k)}^2
		&\leq
		\frac{\eps}{2\kappa} \cdot \frac{m!}{2^{m+1} d^m} \sum_{\abs{\beta}=m} \frac{1}{\beta!}
			\Bigl( \frac{\norm{w^m \partial^\beta f}_{L^2(Q_k)}}{q_m} \Bigr)^2\\
		&\leq
		\frac{\eps}{2\kappa} \sum_{m=0}^\infty \frac{m!}{2^{m+1} d^m} \sum_{\abs{\beta}=m} \frac{1}{\beta!}
			\Bigl( \frac{\norm{w^m \partial^\beta f}_{L^2(Q_k)}}{q_m} \Bigr)^2
		.
	}
	Summing over all bad $Q_k$ with $k\in\cK_0$
	and using \eqref{eq:smoothing-f-delta} then gives
	\eqs{
		\sum_{\substack{k\in \cK_0\\Q_k\, \mathrm{bad}}} \norm{f}_{L^2(Q_k)}^2
		\leq
		\frac{\eps}{2} \cdot D_1^2  \sum_{m=0}^\infty \frac{m!}{2^{m+1} d^m} \sum_{\abs{\beta}=m} \frac{1}{\beta!}
		=
		\frac{\eps}{2} \cdot D_1^2  \sum_{m=0}^\infty \frac{1}{2^{m+1}}
		=
		\frac{\eps}{2} D_1^2
		,
	}
	where we used that $\sum_{\abs{\beta}=m} 1/\beta! = d^m/m!$. This proves \eqref{eq:bad-contribution}.
\end{proof}%

Now, as in \cite{EgidiS-21}, see also \cite{Kovrijkine-01,Kovrijkine-thesis,EgidiV-20}, we use the definition of good elements to
extract a pointwise estimate for the derivatives of $f$.

\begin{lemma}\label{lem:pointwise}
	Let $Q_k$ be good. Then there is $x_k \in Q_k$ such that for all $\beta \in \NN_0^d$ with $\abs{\beta} = m \in \NN_0$ we have
	\be\label{eq:pointwise}
		\abs{\partial^\beta f(x_k)}
			\leq
		\Bigl( \frac{2\kappa}{\eps} \Bigr)^{1/2} \cdot 2^{m+1} d^{m/2} \cdot C(k,m)^{1/2} \cdot
			\frac{\norm{f}_{L^2(Q_k)}}{\sqrt{\abs{Q_k}}}
	\ee
	with
	\be\label{lem:good-const}
		C(k,m)
		=
		q_m^2\sup_{x \in Q_k} w(x)^{-2m}
		.
	\ee
\end{lemma}

\begin{proof}
	Assume that for all $x \in Q_k$ there is $m = m(x) \in \NN_0$ such that
	\[
		\sum_{\abs{\beta}=m} \frac{1}{\beta!}\abs{\partial^\beta f(x)}^2
		>
		\frac{2\kappa}{\eps} \cdot \frac{4^{m+1}d^m}{m!} \cdot C(k,m) \cdot \frac{\norm{f}_{L^2(Q_k)}^2}{\abs{Q_k}}
		.
	\]
	Reordering the terms and summing over all $m \in \NN_0$ in order to get rid of the $x$-dependence, we then obtain
	\be\label{eq:derivatives}
		\frac{\norm{f}_{L^2(Q_k)}^2}{\abs{Q_k}}
		<
		\frac{\eps}{2\kappa} \sum_{m=0}^\infty \frac{m!}{4^{m+1}d^m C(k,m)} \sum_{\abs{\beta}=m} \frac{1}{\beta!}
			\abs{\partial^\beta f(x)}^2
	\ee
	for all $x \in Q_k$. 
	We observe that
	\bes
		\norm{\partial^\beta f}_{L^2(Q_k)}^2
		=
		\norm{w(\cdot)^{-m}w^m\partial^\beta f}_{L^2(Q_k)}^2
		\leq
		\sup_{x \in Q_k} w(x)^{-2m} \cdot \norm{w^m\partial^\beta f}_{L^2(Q_k)}^2
		.
	\ees
	Thus, integrating \eqref{eq:derivatives} over $x \in Q_k$ and using that $Q_k$ is good gives
	\eqs{
		\norm{f}_{L^2(Q_k)}^2
		&<
		\frac{\eps}{2\kappa} \sum_{m=0}^\infty \frac{m!}{4^{m+1}d^m C(k,m)} \sum_{\abs{\beta}=m} \frac{1}{\beta!}
			\norm{\partial^\beta f}_{L^2(Q_k)}^2\\
		&\leq
		\norm{f}_{L^2(Q_k)}^2 \sum_{m=0}^\infty 2^{-m-1}
			=
			\norm{f}_{L^2(Q_k)}^2
		,
	}
	leading to a contradiction. 
	Hence, there is $x_k \in Q_k$ such that for all $\beta \in \NN_0^d$ with $\abs{\beta} = m$ we have
	\[
		\abs{\partial^\beta f(x_k)}^2
		\leq
		\sum_{\abs{\beta}=m} \frac{m!}{\beta!} \abs{\partial^\beta f(x_k)}^2
		\leq
		\frac{2\kappa}{\eps} \cdot 4^{m+1}d^m \cdot C(k,m) \cdot \frac{\norm{f}_{L^2(Q_k)}^2}{\abs{Q_k}}
		,
	\]
	and taking square roots proves the claim.
\end{proof}%

\subsection{The local estimate}

In order to estimate $f$ on each (good) $Q_k$, $k \in \cK_0$, we use a complex analytic local estimate that goes back to
\cite{Nazarov-94,Kovrijkine-01,Kovrijkine-thesis}. It has later been used in \cite{EgidiS-21} and implicitly also in
\cite{EgidiV-20,WangWZZ-19,BeauchardJPS-21,MartinPS-20-arxiv}. We rely here on a particular case of the formulation
in \cite{EgidiS-21}.

\begin{lemma}[{see~\cite[Lemma~3.5]{EgidiS-21}}]\label{lem:local}
	Let $k\in\cK_0$, and suppose that the function $f\restriction_{Q_k} \colon Q_k \to \CC$ has an analytic extension
	$F \colon Q_k + D_{8\rho(x_k)} \to \CC$ with bounded modulus, where
	$D_{8\rho(x_k)} = D(0,8\rho(x_k))\times\dots\times D(0,8\rho(x_k))\subset\CC^d$	is the complex polydisc of radius $8\rho(x_k)$.
	
	Then, for every measurable set $\omega \subset \RR^d$ with $\abs{Q_k\cap\omega} > 0$ we have
	\[
		\Bigl( 24d2^d\cdot \frac{\abs{Q_k}}{\abs{Q_k\cap\omega}}\Bigr)^{1+4\log M_k/\log 2}
		\norm{f}_{L^2(Q_k \cap \omega)}^2
		\geq
		\norm{f}_{L^2(Q_k)}^2
	\]
	with
	\[
		M_k := \frac{\sqrt{\abs{Q_k}}}{\norm{f}_{L^2(Q_k)}} 
		\cdot \sup_{z \in Q_k + D_{8\rho(x_k)}} \abs{F(z)} \ge 1.
	\]
\end{lemma}

On good $Q_k$, the hypotheses of Lemma~\ref{lem:local} are indeed satisfied, and the pointwise estimates \eqref{eq:pointwise} can
be used to obtain a suitable upper bound for the quantity $M_k$.

\begin{lemma}\label{lem:M_k}
	Let $Q_k$ be good.
	Then, the restriction $f\restriction_{Q_k}$ has an analytic extension 
	$F_k \colon Q_k + D_{8\rho(x_k)} \to \CC$, and $M_k$ as in Lemma~\ref{lem:local} satisfies
	\[
		\log M_k
		\leq
		\log (2K') + \frac{1}{2}\log\Bigl( \frac{2\kappa}{\eps} \Bigr) + D^{1/(1-s)}
		,
	\]
	where $K' \geq 1$ is a constant depending only on $s$ and
	\[
		D
		=
		40d^{3/2}\tilde{D}_2^2 R\max\{ r_0 , (1-\eta)^{-1} \}
		.
	\]
\end{lemma}

\begin{proof}
	Let $x_k \in Q_k$ be a point as in Lemma~\ref{lem:pointwise}. 
	For every $z \in x_k + D_{10\rho(x_k)}$ we then have
	\eqs{
		\sum_{\beta \in \NN_0^d} &\frac{\abs{\partial^\beta f(x_k)}}{\beta!} \abs{(z-x_k)^\beta}\\
		&\leq
		\sum_{m\in\NN_0} \sum_{\abs{\beta}=m} \frac{1}{\beta!} \Bigl( \frac{2\kappa}{\eps} \Bigr)^{1/2} 2^{m+1} d^{m/2} C(k,m)^{1/2}
			(10\rho(x_k))^{\abs{\beta}} \frac{\norm{f}_{L^2(Q_k)}}{\sqrt{\abs{Q_k}}}\\
		&=
		2\Bigl( \frac{2\kappa}{\eps} \Bigr)^{1/2} \frac{\norm{f}_{L^2(Q_k)}}{\sqrt{\abs{Q_k}}} \sum_{m\in\NN_0} C(k,m)^{1/2}
			\frac{(20d^{3/2}\rho(x_k))^m}{m!}
		,
	}
	where for the last inequality we again used that $\sum_{\abs{\beta} = m} 1/\beta! = d^m/m!$.
	Taking into account that $Q_k + D_{8\rho(x_k)} \subset x_k + D_{10\rho(x_k)}$ and that $f$ is analytic by Lemma~\ref{lem:analyticity},
	this shows that the Taylor expansion of $f$ around $x_k$ defines an analytic extension $F_k \colon Q_k + D_{8\rho(x_k)} \to \CC$
	of $f$ with bounded modulus and that
	\be\label{eq:M_k}
		M_k
		\leq
		2 \Bigl( \frac{2\kappa}{\eps} \Bigr)^{1/2} \sum_{m=0}^\infty C(k,m)^{1/2} \frac{(20d^{3/2}\rho(x_k))^m}{m!}
		.
	\ee
	
	Now, suppose first that $\abs{x_k} \leq r_0$ with $r_0 \geq 1$ as in \eqref{eq:smallness}. Then,
	\[
		\rho(x_k)
		\leq
		R(1+r_0^2)^{\delta/2}
		\leq 
		R(1+r_0^2)^{1/2}
		\leq
		2Rr_0
		.
	\]
	Using \eqref{eq:M_k},\eqref{lem:good-const}, and the definition of $q_m$, it follows that
	\eqs{
		M_k
		&\leq
		2 \Bigl( \frac{2\kappa}{\eps} \Bigr)^{1/2} \sum_{m=0}^\infty q_m \cdot \sup_{x\in Q_k} \frac{1}{w(x)^m}\cdot\frac{(40d^{3/2}Rr_0)^m}{m!}\\
		&\leq
		2 \Bigl( \frac{2\kappa}{\eps} \Bigr)^{1/2} \sum_{m=0}^\infty \frac{(40d^{3/2}\tilde{D}_2^2 Rr_0)^m}{(m!)^{1-s}}
		,
	}
	where we have taken into account that $w(x) \geq 1$ for all $x\in\RR^d$.

	On the other hand, if $\abs{x_k} \geq r_0$, then for all $x \in Q_k$ we have by \eqref{eq:smallness} the lower bound
	$\abs{x} \geq \abs{x_k} - \rho(x_k) \geq (1-\eta)\abs{x_k} > 0$ and, thus,
	\[
		\frac{\rho(x_k)}{w(x)}
		\leq
		\frac{Rw(x_k)}{w(x)} 
		\leq 2R \Bigl(\frac{\abs{x_k}}{\abs{x}}\Bigr)^\delta
		\leq 
		\frac{2R}{(1-\eta)^{\delta}}
		\leq 
		\frac{2R}{1-\eta}
		.
	\]
	Using again \eqref{eq:M_k} and \eqref{lem:good-const} then gives
	\eqs{
		M_k
		&\leq
		2 \Bigl( \frac{2\kappa}{\eps} \Bigr)^{1/2} \sum_{m=0}^\infty q_m \cdot \sup_{x \in Q_k} \frac{\rho(x_k)^m}{w(x)^m} \cdot
			\frac{(20d^{3/2})^m}{m!}\\
		&\leq
		2 \Bigl( \frac{2\kappa}{\eps} \Bigr)^{1/2} \sum_{m=0}^\infty \frac{(40d^{3/2}\tilde{D}_2^2R/(1-\eta) )^m}{(m!)^{1-s}}
		.
	}
	We conclude that for both cases $\abs{x_k} \leq r_0$ and $\abs{x_k} \geq r_0$ we have
	\eqs{
		M_k
		&\leq
		2 \Bigl( \frac{2\kappa}{\eps} \Bigr)^{1/2} \sum_{m=0}^\infty
			\frac{(40d^{3/2}\tilde{D}_2^2 R\max\{ r_0 , (1-\eta)^{-1} \})^m}{(m!)^{1-s}}\\
		&=
		2 \Bigl( \frac{2\kappa}{\eps} \Bigr)^{1/2} \sum_{m=0}^\infty \frac{D^m}{(m!)^{1-s}}
		.
	}
	We estimate the series using the asymptotics
	\[
		\sum_{m = 0}^\infty\frac{x^m}{(m!)^p}
		=
		\frac{\euler^{px^{1/p}}}{p^{1/2}(2\pi x^{1/p})^{(p-1)/2}}\Bigl\{ 1 + O\Bigl( \frac{1}{x^{1/p}} \Bigr) \Bigr\}
		\qquad
		(p \in (0,4],\ x \to \infty)
	\]
	derived in \cite[Eq.~(8.07)]{Olver-97}.
	Thereby,
	\[
		\sum_{m=0}^\infty \frac{D^m}{(m!)^{1-s}} 
		\leq 
		K'\euler^{D^{1/(1-s)}}
		,
	\]
	where $K'$ is a constant depending only on $s$.
	Hence, 
	\[
		M_k 
		\leq
		2K'\Bigl( \frac{2\kappa}{\eps} \Bigr)^{1/2}\euler^{D^{1/(1-s)}}
	\]
	and therefore
	\[
		\log M_k
		\leq
		\log (2K') + \frac{1}{2}\log\Bigl( \frac{2\kappa}{\eps} \Bigr) + D^{1/(1-s)}
		.\qedhere
	\]
\end{proof}%

We are now in position to prove our main result.

\begin{proof}[Proof of Theorem~\ref{thm:uncertainty}]
	By hypothesis we have $\abs{Q_k} / \abs{Q_k\cap\omega} \leq 1/\gamma$.
	Combining this with Lemma~\ref{lem:local} and the estimate for $\log M_k$ derived in Lemma~\ref{lem:M_k}, we obtain for
	all good $Q_k$ that
	\[
		\Bigl( \frac{24d\cdot 2^d}{\gamma} \Bigr)^{5+\bigl(4\log(K')+2\log(\frac{2\kappa}{\eps}) + 4D^{1/(1-s)}\bigr)/\log 2}
		\norm{f}_{L^2(Q_k \cap \omega)}^2
		\geq
		\norm{f}_{L^2(Q_k)}^2
		.
	\]
	In particular, 
	\be
		\Bigl(\frac{1}{\gamma}\Bigr)^{K''\cdot\bigl(1+\log\frac{1}{\eps} + D_2^{2/(1-s)}\bigr)}
		\norm{f}_{L^2(Q_k \cap \omega)}^2
		\geq
		\norm{f}_{L^2(Q_k)}^2
		,
	\ee
	where $K'' \geq 1$ is constant depending on $R,r_0,\eta,\nu,s$, and the dimension $d$.
	Summing over all good $Q_k$ gives
	\eqs{
		\sum_{k\colon Q_k\ \text{good}} \norm{f}_{L^2(Q_k)}^2
		&\leq 
		\Bigl(\frac{1}{\gamma}\Bigr)^{K''\cdot\bigl(1+\log\frac{1}{\eps} + D_2^{2/(1-s)}\bigr)}
			\sum_{k\colon Q_k\ \text{good}} \norm{f}_{L^2(Q_k\cap\omega)}^2
		\\
		&\leq 
			\kappa\Bigl(\frac{1}{\gamma}\Bigr)^{K''\cdot\bigl(1+\log\frac{1}{\eps} + D_2^{2/(1-s)}\bigr)}\norm{f}_{L^2(\omega)}^2
		.
	}
	Together with Lemma~\ref{lem:good-bad} this proves the theorem with $K \leq K''\log(1/\gamma) + \log \kappa$.
\end{proof} 

A slight adaptation of the proof of Theorem~\ref{thm:uncertainty} allows us to consider a variable density $\gamma = \gamma(x)$ with
a polynomial decay. 

\begin{theorem}\label{thm:uncertainty-decay}
	Let $D_1,D_2,\mu,\nu$, and $\delta$ be as in Theorem~\ref{thm:uncertainty} above, and suppose that $f$
	satisfies \eqref{eq:smoothing-f}.
	Let $\omega \subset \RR^d$ be any measurable set satisfying 
	\bes
		\frac{\abs{B(x,\rho(x))\cap \omega}}{\abs{B(x,\rho(x))}}
		\geq 
		\frac{\gamma_0}{1+\abs{x}^a}
		\quad \text{for all} \quad 
		x\in\RR^d
	\ees
	with some $\gamma_0\in (0,1)$ and some $a > 0$.
	
	Then 
	\be\label{eq:uncertainty-decay}
		\norm{f}_{L^2(\RR^d)}^2
		\leq 
		\euler^{K\cdot (1+\log\frac{1}{\eps}+D_2^{2/(1-s)})^2} \norm{f}_{L^2(\omega)}^2 + \eps D_1^2
		,
	\ee
	where $K \geq 1$ is a constant depending on $\gamma_0,R,r_0,\eta,\nu,s,a$, and the dimension $d$.
\end{theorem} 

\begin{proof}
	We have $\abs{Q_k} / \abs{Q_k\cap\omega} \leq (1+\abs{y_k}^a) / \gamma_0$, and it is easily checked that
	\[
		1+\abs{y_k}^a
		\leq 
		\frac{2^{1+a}r_0^a}{(1-\eta)^a}r^a
	\]
	for all $k\in\cK_0$. Since $r$ is as in \eqref{eq:choice-r}, this shows 
	\bes
		\log( 1+\abs{y_k}^a )
		\leq 
		\tilde{K}\Bigl(1 + \log\frac{1}{\eps} + \log D_2\Bigr)
		,
	\ees
	where $\tilde{K} \geq 1$ is a constant depending on $r_0, \eta$, and $a$. In light of the inequality
	$\log D_2 \leq D_2^{2/(1-s)}$, the theorem now follows in the same way as Theorem~\ref{thm:uncertainty}.
\end{proof} 

\begin{remark}\label{rem:ucp}
	The exponent in \eqref{eq:uncertainty-decay} depends on $\eps$ essentially by $(\log(1/\eps))^2$, as compared to just $\log(1/\eps)$ in
	\eqref{eq:uncertainty} of Theorem~\ref{thm:uncertainty}. To the best of our knowledge, the proof of the observability estimate from
	\cite{Martin-21,Miller-10} does not work with the kind of dependence in \eqref{eq:uncertainty-decay}.
\end{remark} 

\appendix

\section{Analyticity}

The following result establishes that Gelfand-Shilov smoothing effects as in 
\eqref{eq:smoothing} guarantee analyticity of the functions in the range of the semigroup. 
This is required in order to follow the complex analytic approach discussed
in the main text.

\begin{lemma}\label{lem:analyticity}
	Let $f \in C^\infty(\RR^d)$ be such that
	\[
		\norm{\partial^\beta f}_{L^2(\RR^d)}
		\leq
		C_1 C_2^{\abs{\beta}} \beta!
		\quad 
		\text{for all}\quad
		\beta \in \NN_0^d
	\]
	with some constants $C_1,C_2 > 0$. 
	Then, $f$ is analytic in $\RR^d$.
\end{lemma}

\begin{proof}
	Choose $\sigma \in (0,1]$ with $2C_2\sigma < 1$. Let $y \in \RR^d$, and let $B = B(y,\tau)$ be an open ball around $y$ with
	$\tau < \sigma/d$. We show that the Taylor series of $f$ around $y$ converges in $B$ and agrees with $f$ there. To this end, it
	suffices to establish
	\be\label{eq:toshow}
		\sum_{\alpha \in \NN_0} \frac{\norm{ \partial^\alpha f }_{L^\infty(B)}}{\alpha!} \tau^{\abs{\alpha}}
		<
		\infty
		;
	\ee
	cf.~\cite[Theorem~2.2.5 and Proposition 2.2.10]{KrantzP-02}.
	
	We proceed similarly as in the proof of \cite[Lemma~3.2]{EgidiS-21}: Since $B$ satisfies the cone condition, by Sobolev
	embedding there exists a constant $c > 0$, depending	only on $\tau$ and the dimension, such that
	$\norm{g}_{L^\infty(B)}\leq c \norm{g}_{W^{d,2}(B)}$ for all $g \in W^{d,2}(B)$,
	see, e.g.,~\cite[Theorem~4.12]{AdamsF-03}. Applying this to $g = \partial^\alpha f\restriction_B$ with
	$\abs{\alpha} = m \in \NN_0$, we obtain
	\eqs{
		\norm{ \partial^\alpha f }_{L^\infty(B)}^2
		&\leq
		c^2 \norm{\partial^\alpha f}_{W^{d,2}(B)}^2
			\leq c^2 \norm{\partial^\alpha f}_{W^{d,2}(\RR^d)}^2\\
		&=
		c^2 \sum_{\abs{\beta} \leq d} \norm{\partial^{\beta+\alpha}f}_{L^2(\RR^d)}^2
			\leq
			c^2 \sum_{k=m}^{m+d}\sum_{\abs{\beta} = k} \norm{\partial^{\beta}f}_{L^2(\RR^d)}^2
		.
	}
	Taking the square root and using the hypothesis gives
	\[
		\norm{ \partial^\alpha f }_{L^\infty(B)}
		\leq
		c \sum_{k=m}^{m+d}\sum_{\abs{\beta} = k} \norm{\partial^{\beta}f}_{L^2(\RR^d)}
		\leq
		c C_1 \sum_{k=m}^{m+d} C_2^k \sum_{\abs{\beta} = k} \beta!
		.
	\]
	We clearly have
	\[
		\sum_{\abs{\beta} = k} \beta!
		\leq
		k! \cdot \#\{ \beta \in\NN_0^d \mid \abs{\beta}=k \}
		\leq
		k! 2^{k+d-1}
		.
	\]
	In view of the choice of $\sigma$, we thus further estimate
	\eqs{
		\norm{ \partial^\alpha f }_{L^\infty(B)}
		&\leq
		2^{d-1}c C_1 \sum_{k=m}^{m+d} (2C_2)^k k!
			\leq
			2^{d-1}c C_1 \frac{(m+d)!}{\sigma^{m+d}} \sum_{k=m}^{m+d} ( 2C_2\sigma )^k\\
		&\leq
		\frac{2^{d-1}c C_1}{\sigma^d} \sum_{k=0}^{\infty} ( 2C_2\sigma )^k \cdot \frac{(m+d)!}{\sigma^m}
			=:C_0 \cdot \frac{(m+d)!}{\sigma^m}
		.
	}
	Now,
	\eqs{
		\sum_{\abs{\alpha} = m} \frac{\norm{ \partial^\alpha f }_{L^\infty(B)}}{\alpha!}
		&\leq
		C_0 \frac{(m+d)!}{\sigma^m} \sum_{\abs{\alpha}=m} \frac{1}{\alpha!}
			=
			C_0 \Bigl( \frac{d}{\sigma} \Bigr)^m \frac{(m+d)!}{m!}\\
		&\leq
		C_0 \Bigl( \frac{d}{\sigma} \Bigr)^m (m+d)^d
		,
	}
	and since $\tau$ is chosen such that $d\tau / \sigma < 1$, this shows \eqref{eq:toshow} and, hence, completes the proof.
\end{proof}%

\providecommand{\bysame}{\leavevmode\hbox to3em{\hrulefill}\thinspace}
\providecommand{\MR}{\relax\ifhmode\unskip\space\fi MR }
\providecommand{\MRhref}[2]{%
  \href{http://www.ams.org/mathscinet-getitem?mr=#1}{#2}
}
\providecommand{\href}[2]{#2}


\begin{thebibliography}{10}

\bibitem{AdamsF-03}
R.~A. {Adams} and J.~J.~F. {Fournier}, \emph{{Sobolev Spaces}}, New York, NY:
  Academic Press, 2003.

\bibitem{Alphonse-20b}
P.~Alphonse, \emph{Null-controllability of evolution equations associated with
  fractional {Shubin} operators through quantitative {Agmon} estimates},
  Preprint: arxiv.org/2012.04374.

\bibitem{Alphonse-20}
\bysame, \emph{Quadratic differential equations: {P}artial {G}elfand-{Shilov}
  smoothing effect and null-controllability}, Journal of the Institute of
  Mathematics of Jussieu (2020), 1–53.

\bibitem{BeauchardJPS-21}
K.~Beauchard, P.~Jaming, and K.~Pravda-Starov, \emph{Spectral estimates for
  finite combinations of {H}ermite functions and null-controllability of
  hypoelliptic quadratic equations}, Studia Math. \textbf{260} (2021), no.~1,
  1--43.

\bibitem{BeauchardPS-18}
K.~Beauchard and K.~Pravda-Starov, \emph{Null-controllability of hypoelliptic
  quadratic differential equations}, J. \'Ec. polytech. Math. \textbf{5}
  (2018), 1--43.

\bibitem{DickeRST-20}
A.~Dicke, C.~Rose, A.~Seelmann, and M.~Tautenhahn, \emph{Quantitative unique
  continuation for spectral subspaces of {S}chr\"odinger operators with
  singular potentials}, Preprint: \url{https://arxiv.org/abs/2011.01801}.

\bibitem{DickeSV-22}
A.~Dicke, A.~Seelmann, and I.~Veseli\'{c}, \emph{Control problem for quadratic
  parabolic differential equations with sensor sets of finite volume or
  anisotropically decaying density}, Preprint: \url{https://arxiv.org/abs/2201.02370}.

\bibitem{DickeSV-21}
\bysame, \emph{Uncertainty principle for {H}ermite functions and
  null-controllability with sensor sets of decaying density}, submitted.

\bibitem{EgidiS-21}
M.~Egidi and A.~Seelmann, \emph{An abstract {L}ogvinenko-{S}ereda type theorem
  for spectral subspaces}, J. Math. Anal. Appl. \textbf{500} (2021), no.~1,
  125149.

\bibitem{EgidiV-20}
M.~{Egidi} and I.~{Veseli\'c}, \emph{{Scale-free unique continuation estimates
  and Logvinenko-Sereda theorems on the torus}}, {Ann. Henri Poincar\'e}
  \textbf{21} (2020), no.~12, 3757--3790.

\bibitem{GallaunST-20}
D.~{Gallaun}, C.~{Seifert}, and M.~{Tautenhahn}, \emph{{Sufficient criteria and
  sharp geometric conditions for observability in Banach spaces}}, {SIAM J.
  Control Optim.} \textbf{58} (2020), no.~4, 2639--2657.

\bibitem{Kovrijkine-thesis}
O.~Kovrijkine, \emph{Some estimates of {F}ourier transforms}, ProQuest LLC, Ann
  Arbor, MI, 2000, Thesis (Ph.D.)--California Institute of Technology.

\bibitem{Kovrijkine-01}
\bysame, \emph{Some results related to the {L}ogvinenko-{S}ereda theorem},
  Proc. Amer. Math. Soc. \textbf{129} (2001), no.~10, 3037--3047.

\bibitem{KrantzP-02}
S.~G. {Krantz} and H.~R. {Parks}, \emph{{A Primer of Real Analytic
  Functions.}}, Boston, MA: Birkh\"auser, 2002.

\bibitem{Martin-21}
J.~Martin, \emph{Uncertainty principles in {G}elfand-{S}hilov spaces and
  null-controllability}, Preprint: \url{https://arxiv.org/abs/2112.01788}.

\bibitem{MartinPS-20-arxiv}
J.~Martin and K.~Pravda-Starov, \emph{Spectral inequalities for combinations of
  {H}ermite functions and null-controllability for evolution equations enjoying
  {G}elfand-{S}hilov smoothing effects}, Preprint: \url{https://arxiv.org/abs/2007.08169v2}.

\bibitem{Mattila-99}
P.~{Mattila}, \emph{{Geometry of Sets and Measures in Euclidean Spaces.
  Fractals and Rectifiability. 1st paperback ed}}, 1st paperback ed. ed.,
  vol.~44, Cambridge: Cambridge University Press, 1999.

\bibitem{Miller-10}
L.~Miller, \emph{A direct {L}ebeau-{R}obbiano strategy for the observability of
  heat-like semigroups}, Discrete \& Continuous Dynamical Systems - B
  \textbf{14} (2010), no.~4, 1465--1485.

\bibitem{MoyanoL-19}
I.~Moyano and G.~Lebeau, \emph{Spectral {I}nequalities for the {S}chr\"odinger
  operator}, Preprint: \url{https://arxiv.org/abs/1901.03513}.

\bibitem{NakicTTV-20}
I.~Naki\'c, M.~T\"aufer, M.~Tautenhahn, and I.~Veseli\'c, \emph{Sharp estimates
  and homogenization of the control cost of the heat equation on large
  domains}, ESAIM Control Optim. Calc. Var. \textbf{26} (2020), no.~54, 26
  pages.

\bibitem{NakicTTVS-20}
I.~Naki{\'c}, M.~T{\"a}ufer, M.~Tautenhahn, and I.~Veseli{\'c}, \emph{Unique
  continuation and lifting of spectral band edges of {S}chr{\"o}dinger
  operators on unbounded domains}, Journal of Spectral Theory \textbf{10}
  (2020), no.~3, 843--885, With an Appendix by Albrecht Seelmann.

\bibitem{NazarovSV-04}
F.~{Nazarov}, M.~{Sodin}, and A.~{Volberg}, \emph{{Lower bounds for
  quasianalytic functions. I: How to control smooth functions}}, {Math. Scand.}
  \textbf{95} (2004), no.~1, 59--79.

\bibitem{Nazarov-94}
F.~L. Nazarov, \emph{Local estimates for exponential polynomials and their
  applications to inequalities of the uncertainty principle type}, Algebra i
  Analiz \textbf{5} (1993), no.~4, 3--66.

\bibitem{Olver-97}
F.~W.~J. Olver, \emph{{Asymptotics and special functions}}, Wellesley, MA: A K Peters, reprint edition, 1997.

\bibitem{TenenbaumT-11}
G.~Tenenbaum and M.~Tucsnak, \emph{On the null-controllability of diffusion
  equations}, ESAIM Control Optim. Calc. Var. \textbf{17} (2011), no.~4,
  1088--1100.

\bibitem{WangWZZ-19}
G.~Wang, M.~Wang, C.~Zhang, and Y.~Zhang, \emph{Observable set, observability,
  interpolation inequality and spectral inequality for the heat equation in
  $\mathbb{R}^n$}, J. Math. Pures Appl. (9) \textbf{126} (2019), 144--194.

\end{thebibliography}
\end{document}